\documentclass[11pt]{article}
\usepackage{amsmath,amsfonts,amsthm,amssymb,mathrsfs,mathtools,ulem}
\usepackage{graphicx}
\usepackage[usenames,dvipsnames]{color}
\usepackage{float}
\usepackage{graphicx}
\usepackage{subfigure}
\usepackage{caption}
\usepackage{cite}
\usepackage{url}
\usepackage{caption}
\usepackage{epstopdf}
\usepackage{hyperref}
\usepackage[nice]{nicefrac}
\usepackage{color}

\newcommand{\dint}{\displaystyle\int}

\newcommand\redout{\bgroup\markoverwith
{\textcolor{red}{\rule[0.5ex]{2pt}{0.8pt}}}\ULon}

\theoremstyle{plain}
\newtheorem{theorem}{Theorem}[section]

\newtheorem{corollary}[theorem]{Corollary}
\newtheorem{lemma}[theorem]{Lemma}
\newtheorem{proposition}[theorem]{Proposition}

\theoremstyle{definition}
\newtheorem{definition}[theorem]{Definition}

\theoremstyle{remark}
\newtheorem{remark}[theorem]{Remark}

\numberwithin{equation}{section}
\numberwithin{theorem}{section}

\usepackage{a4wide}
\usepackage{geometry}
 \title{Images of Fractional Brownian motion with deterministic drift: Positive Lebesgue  measure and non-empty interior}
 \date{}
\begin{document}
	\maketitle
 
 \begin{center}
     
   \author{ MOHAMED ERRAOUI\\
   Department of mathematics, Faculty of science El jadida,\\ Chouaïb Doukkali University,
   Morocco \\
  e-mail\textup{: \texttt{erraoui@uca.ac.ma}}}

\end{center}
\begin{center}
         \author{ YOUSSEF HAKIKI\footnote{Supported by National Center for Scientific and Technological Research (CNRST)}\\
   Department of mathematics, Faculty of science Semlalia,\\ Cadi Ayyad University, 2390 Marrakesh, Morocco \\
  e-mail\textup{: \texttt{youssef.hakiki@ced.uca.ma}}}
 
 \end{center} 

	\begin{abstract}
		Let $B^{H}$ be a fractional Brownian motion in $\mathbb{R}^{d}$    
		of Hurst index $H\in\left(0,1\right)$, $f:\left[0,1\right]\longrightarrow\mathbb{R}^{d}$
		a Borel function and $A\subset\left[0,1\right]$ a Borel set. 
				 We provide	sufficient conditions for the image $(B^{H}+f)(A)$ to have a positive Lebesgue measure or to have a non-empty interior. 
		This is done through the study of the properties of the density of the occupation measure of $(B^{H}+f)$. Precisely, we prove that if the parabolic Hausdorff dimension of the graph of $f$ is greater than $Hd$, then the density is a square integrable function. If, on the other hand, the Hausdorff dimension of $A$ is greater than $ Hd$, then it even admits a continuous version. This allows us to establish the result already cited.
	\end{abstract}

\vspace{0,2cm}


\section{Introduction}
Many sets from the trajectories of the fractional Brownian motion
are somehow random fractals. These sets have an interesting geometric
structure at all scales. Among the most studied in the literature,
it is worth mentioning images and graphs of a fractional Brownian motion. For standard reference we refer to Adler \cite{A},
Kahane \cite{Kahane}. So, we start by recalling the exact Hausdorff dimension of the image $B^{H}(A)$ where $B^{H}$ is a fractional Brownian motion in $\mathbb{R}^{d}$ of Hurst index $H\in \left( 0,1\right)$ and $A\subset\mathbb{R}$ is a Borel set. It is visibly that $B^{H}(A)$ is
a random set in $\mathbb{R}^{d}$. There has been considerable interest
in the study of its geometric and arithmetic properties.
It is proved that almost surely $$\dim B^{H}(A)=\min\left\{ \dfrac{\dim(A)}{H},d\right\} .$$
We write $\dim(\centerdot)$ for the Hausdorff dimension, and refer to \cite{F} and \cite{Kahane} for its definition and basic properties. Several works have focused strongly on determining whether $B^{H} (A)$ is a.s. a set of positive Lebesgue measure or a Salem set by comparing  $\dim(A)$ and $Hd$, see \cite{Kahane} and \cite{SX}. Two distinct
cases are clearly apparent : $\dim(A)>Hd$ or $\dim(A)<Hd$. 

In the first case, a proof based on Fourier transform due to Kahane
(see \cite{Kahane})  provided a preliminary result, namely $B^{H}(A)$ a.s. has positive
$d$-dimensional Lebesgue measure. Pitt \cite{Pi} and Kahane  \cite{Kahane} strengthened this result by proving that $ B ^ {H} (A) $ a.s. has non-empty interior. However, it should be noted that, for the one dimensional Brownian motion the fact that $B^{H}(A)$ has non-empty interior when $\dim(A)>1/2$ is due to Kaufman in \cite{Kaufman}. In the case $\dim(A)<Hd$, Kahane \cite{Kahane} established that $B^{H}(A)$ is a.s. a Salem set. The reader might refer
to the book of Kahane \cite{Kahane} for the precise
definition and the basic properties of Salem set since it will not be treated here.

Recently Peres and Sousi \cite{Peres&Sousi FBM}
studied fractal properties of  images and graphs of $B^{H}+f$ where $f:[0,1] \rightarrow \mathbb{R}^d$
is a Borel measurable function. They expressed, for a Borel set $A\subset [0,1]$, the dimension of
the image set $\left( B^{H}+f\right)(A)$ in terms of the so-called parabolic
Hausdorff dimension of the graph of $f$ restricted to $A$ denoted by $\dim_{\Psi, H}(Gr_A(f))$, see Definition \ref{H-parabolic dimension}. 
Precisely, the authors stated that almost surely $$
\dim\left((B^{H}+f)(A)\right)=\min \left( \dfrac{\dim_{\Psi, H}(Gr_A(f))}{H}, d\right).$$ 

It is important to note that the parabolic Hausdorff
dimension has first been considered by Taylor and Watson  \cite{Taylor&Watson}
for the study of polar sets of the heat equation. It has also been shown to be useful for analyzing the geometry of the images of Brownian motion, see Khoshnevisan and Xiao \cite{KX}. Once again
when $B^{H}$ is the multifractional Brownian motion (mBm), the parabolic
Hausdorff dimension again happens to be the relevant tool to describe
the geometry of its graph and the images of sets, cf \cite{Bal}. For all these reasons, it is therefore quite natural to ask the following question:
what are the properties of the set $\left( B^{H}+f\right)(A)$ that can be explored deeply from the informations provided by the parabolic dimension of the graph of $f$?
This issue will be the main goal of this work, which we consider as a continuation of  \cite{Peres&Sousi FBM}. 

In Section 2, first we
give a comparison result for parabolic Hausdorff dimensions
for different parameters. We look also seek to determine the exact value of $\dim_{\Psi,H}(Gr_{A}(f))$.
It is already becoming apparent that calculation of parabolic Hausdorff
dimensions can be a little involved, even for simple sets. For this,
as a first step we establish lower and upper bounds of $\dim_{\Psi,H}(Gr_{A}(f))$,
that take into account the Hölder continuity and $\dim(A)$. We note that
the lower bound of $\dim_{\Psi,H}(Gr_{A}(f))$ can be seen as a kind
of the Mastrand projection theorem. Secondly, in order to give sharp bounds, we looked for functions with known graphs, such as the paths of Hölder's continuous stochastic process. Precisely, we have considered the paths of the fractional
Brownian motion with a Hurst parameter less than $H$. For this class we were able to calculate the exact value of the desired parabolic Hausdorff dimension.

In this section 3, our investigation focuses on the fractional Brownian
motion with drift $\left(B^{H}+f\right)$ and turned to the question
of whether the image $\left(B^{H}+f\right)(A)$ admit a positive Lebesgue
measure or have a non-empty interior. The question we are examining
arises from the subject of the polar functions of the planar Brownian
motion initiated by Graversen \cite{Graversen} and explored in a more general context,
including both the spatial dimension and the Hölder property, by Le
Gall \cite{LG} and Antunović et al \cite{Vermesi}. Here the idea is to construct, under the condition
$\dim_{\Psi,H}(Gr_{A}(f))>Hd$, a probability $\nu$ carried by $A$ such that its image measure $\mu$ under the
mapping $A\ni t\longrightarrow\left(B^{H}+f\right)(t)$, commonly known as occupation measure, which is supported
by $\left(B^{H}+f\right)(A)$
has a density with respect to the Lebesgue measure.  This latter density is known as occupation density and interpreted as local time over the set $A$.
The idea of introducing local times as densities of occupation measure
has been fruitful in a variety of contexts. Important papers in this
direction are Geman and Horowitz \cite{GH } and Geman, Horowitz and
Rosen \cite{GHR}. Our first result indicates that image $\left(B^{H}+f\right)(A)$
admit a positive Lebesgue measure since the occupation density is
a square integrable function. In terms of the second part of our question,
we will follow the approach used in Kahane  \cite{Kahane} and Pitt \cite{Pi}. The keystone of this
approach is to reduce the problem of covering an open set to the existence
of a continuous occupation density function. To do this, we'll draw
upon on the condition used in the case without drift, which is a little
stronger than that already assumed, namely $\dim(A)>Hd$, see Propositon \ref{lemma equality dimensions}. 
 It should be noted that the results of this section are indeed extensions of those established by Pitt and Kahane for fractional Brownian motion to the case of fractional Brownian motion with variable drift through the comparison between $\dim_{\Psi, H}(Gr_A(f))$ and $Hd$.

Our motivation for section 4 is as follows: in his note \cite{LG} on Graversen's
conjecture on the polar functions for the planar Brownian motion,
Le Gall asked if for each $\gamma<1/d$ there exist $\gamma$-Hölder
continuous functions which are non-polar for $d$-dimensional Brownian
motion $B^{\nicefrac{1}{2}}$. An answer is given by Antunović et
al \cite{Vermesi} in terms of modifications of the standard Hilbert curve perturbed
by Brownian motion. Precisely, they showed that
for any $\gamma<1/d$, there exists a $\gamma$-Hölder continuous
function $f:[0,1]\longrightarrow\mathbb{R}^{d}$ for which the set
$\left(B^{\nicefrac{1}{2}}-f\right)\left(\left[0,1\right]\right)$
covers an open set almost surely. On the other hand, as mentioned
before when $\dim(A)<Hd$, $B^{H}(A)$ is a.s. a Salem set. It is
natural to ask whether for each $\alpha\in\left(0,\dim(A)/d\right)$
there exist $\alpha$-Hölder continuous function $f:[0,1]\rightarrow\mathbb{R}^{d}$
for which the range $(B^{H}+f)(A)$ has a non-empty interior a.s.? 

Since $\dim(A)>\alpha d$ then the previous case allows us to assert that 
 the set $\left(B^{\alpha}+f\right)(A)$ have a non-empty interior
where $B^{\alpha}$ is another fractional Brownian motion with Hurst
index possibly defined on different
probability space. Our idea is to consider $B^{H}$ as a drift of
$B^{\alpha}$. This induces us to bring them together on the same
space while preserving their distributions. The best way to do this
is to work on the product space and to consider processes on this
latter. 

Here are some notations that we will use throughout this work. $\lVert\cdot\rVert_{\infty}$  denotes  the maximum norm on $\mathbb{R}^d$ . $ \left\langle \cdot, \cdot  \right\rangle $ and $\lVert\cdot\rVert$ are the ordinary scalar product and the Euclidean
norm in  $\mathbb{R}^{d}$ respectively.  If $(E,\rho)$ is a metric space, then the Borel $\sigma$-algebra over $E$ will be denoted by $\mathcal{B}(E)$. We denote by $\dim A$ the Hausdorff dimension of a set $A \subset \mathbb{R}$.
For a function $f :\mathbb{R}_+ \rightarrow \mathbb{R}^d$,  $Gr_A(f)=\{(t,f(t)) : t\in A\}$ is  the graph of $f$ over the set $A$. We will use $C,C_1,\ldots,C_4$ to denote unspecified positive finite constants which may not necessarily be the same in each occurrence.

\section{Preliminaries on Parabolic Hausdorff dimension}

Let $B^{H}_0=\left\lbrace B^{H}_0(t),t\geq 0\right\rbrace $ be a real-valued fractional Brownian motion  of Hurst index $H$ defined on a complete  probability space $(\Omega,\mathcal{F},\mathbb{P})$, i.e. a real valued Gaussian process with stationary increments and covariance function given by $$\mathbb{E}(B^{H}_0(s)B^{H}_0(t))=\frac{1}{2}(|t|^{2H}+|s|^{2H}-|t-s|^{2H}).$$
 Let  $B^{H}_1,...,B^{H}_d$ be $d$ independent copies of  $B^{H}_0$, then the stochastic process $B^{H}=\left\lbrace B^{H}(t),t\geq 0\right\rbrace $  given by  
\begin{align*}
	B^{H}(t)=(B^{H}_1(t),....,B^{H}_d(t)) ,
\end{align*}
is called a $d$-dimensional fractional Brownian motion of Hurst index $H\in (0,1)$. 

We start by giving the definition of the parabolic Hausdorff dimension.
\begin{definition}\label{H-parabolic dimension}
	Let $F\subset \mathbb{R}_{+}\times \mathbb{R}^d$ and $H\in (0,1)$. For all $\beta>0$ the $H$-parabolic $\beta$-dimensional Hausdorff content is defined by
	\begin{align}
	\Psi_{H}^{\beta}(F)=\inf \left\{\sum_{j} \delta_{j}^{\beta} : F \subseteq \underset{j}{\cup}\left[a_{j}, a_{j}+\delta_{j}\right] \times\left[b_{j, 1}, b_{j, 1}+\delta_{j}^{H}\right] \times \ldots \times\left[b_{j, d}, b_{j, d}+\delta_{j}^{H}\right]\right\}\label{H-parabolic Hausdorff content}
	\end{align}
	where the infimum is taken over all covers of $F$ by rectangles of the form given above. The $H$-parabolic Hausdorff dimension is then defined to be $$
	\operatorname{dim}_{\Psi, H}(F)=\inf \left\{\beta : \Psi_{H}^{\beta}(F)=0\right\}.
	$$
\end{definition}
\begin{remark}
Let $\rho_H$ be the metric defined on $\mathbb{R}_+\times\mathbb{R}^d$ by 
\begin{equation}
\rho_H((s,x),(t,y))=\max\{|s-t|^H,\lVert x-y \rVert_{\infty}\} \quad \forall (s,x), (t,y)\in \mathbb{R}_+\times \mathbb{R}^d.\label{norm}
\end{equation}
 We define the $\beta$-dimensional Hausdorff content as
	\begin{align}
\mathcal{H}_{\rho_H}^{\beta}(F)=\inf \left\{\sum_{j} diam(U_{j})^{\beta} : F \subseteq \underset{j}{\cup} \, U_{j} \right\}\label{H-Hausdorff content}
\end{align}
where $\left\lbrace U_{j} \right\rbrace $ is a countable cover of $F$ by any sets and $diam(U_{j}) $ denotes the diameter of a set $U_{j}$ relatively to the metric $\rho_H$.
For any $F\subseteq\mathbb{R}_+\times\mathbb{R}^d$,
the Hausdorff dimension, in the metric $\rho_H$, of $F$ is defined by
$$\dim_{\rho_H}(F)=\inf \left\{\beta : \mathcal{H}_{\rho_H}^{\beta}(F)=0\right\}.$$
 Since the diameter of the set $\left[a_{j}, a_{j}+\delta_{j}\right] \times\left[b_{j, 1}, b_{j, 1}+\delta_{j}^{H}\right] \times \ldots \times\left[b_{j, d}, b_{j, d}+\delta_{j}^{H}\right] $ is $\delta_{j}^{H}$ then it can be seen without difficulty that for any $\beta >0$ and $F\subseteq\mathbb{R}_+\times\mathbb{R}^d$ we have 
\begin{align}
\Psi_H^{\beta}(F)=\mathcal{H}_{\rho_H}^{\beta/H}(F).\label{equality content}
\end{align} 
Hence we obtain
\begin{equation}\label{dim-comp}
\dim_{\Psi,H}(F)=H\times \dim_{\rho_H}(F).
\end{equation}
\end{remark}

The following proposition relates $\beta$-dimensional capacity to the $H$-parabolic Hausdorff dimension.

\begin{proposition}\label{capacity approach to dim_H}
	Let $F\subset \mathbb{R}_+\times\mathbb{R}^d$ be a compact set. Then we have 
	\begin{align}
		\dim_{\Psi,H}(F)=\sup\{\beta: \mathcal{C}_{\rho_H,\beta/H}(F)>0\}=\inf\{\beta: \mathcal{C}_{\rho_H,\beta/H}(F)=0\},
	\end{align}
	where $\mathcal{C}_{\rho_H,\beta}(.)$  is the $\beta$-capacity on the metric space $(\mathbb{R}_+\times\mathbb{R}^d,\rho_H)$  defined by 
	\begin{equation}
		\mathcal{C}_{\rho_H,\beta}(F)=\left[\inf _{\mu \in \mathcal{P}(F)} \int_{\mathbb{R}_+\times\mathbb{R}^{d}} \int_{\mathbb{R}_+\times\mathbb{R}^{d}} \frac{\mu(d u) \mu(d v)}{(\rho_H(u,v))^{\beta}} \right]^{-1}.\label{parabolic capacity}
	\end{equation}
	Here $\mathcal{P}(F)$ is the family of probability measure carried by $F$.
\end{proposition}
\begin{proof}
The proof is the same as in the Euclidean case, see for example Theorem 3.4.2 in \cite{BiPe} or Theorem 4.32 in \cite{MP} .
\end{proof}

The next theorem is the analogue of Frostman’s theorem for parabolic Hausdorff dimension. The statement can be found in Taylor and Watson in \cite[Lemma 4]{Taylor&Watson} and the proof follows along the same lines as the proof of usual Frostman’s theorem. 

\begin{theorem}(Frostman's theorem)\label{Frostman parabolic theorem}
	Let $F$ a Borel set in $\mathbb{R}_+\times\mathbb{R}^d$. If $\dim_{\Psi, H}(F)>\kappa$, then there exists a Borel probability measure $\mu$ supported on $F$, and a constant $C>0$, such that 
	\begin{equation}
			\mu\left([a, a+\delta] \times \prod_{j=1}^d\left[b_{j}, b_{j}+\delta^{H}\right]\right) \leq C \delta^{\kappa},\label{ineq Frostman parabolic}
	\end{equation}
 for any $\left(a,b_{1},\cdots,b_{d} \right)\in  \mathbb{R}_+\times\mathbb{R}^d$ and $\delta
	> 0$.
\end{theorem}

Now we give a comparison result for the Hausdorff parabolic dimensions with different parameters. We note that it is a generalization of Lemma 5.4 in \cite{Bal} to the $d$-dimensional case. However, there is a flaw (may be a misprint) in the lower bound of Lemma 5.4. The following proposition corrects it  and improves the lower and upper bounds.
\begin{proposition}\label{camparison dimensions}
	Let $F\subset \mathbb{R}_+\times \mathbb{R}^d$ and $H,H'\in (0,1)$ such that $H<H'$. Then we have 
	\begin{equation}
\dim_{\Psi,H}(F)\vee 	\left( \dfrac{H'}{H}\dim_{\Psi,H}(F)+1-\dfrac{H'}{H}\right) \leq \dim_{\Psi, H'}(F) \leq \left( \frac{H'}{H}\dim_{\Psi, H}(F)\right) \wedge \left( \dim_{\Psi, H}(F)+(H'-H)d\right) .\label{comparison dim_H dim_K}
	\end{equation}	 
An equivalent reformulation is
\[
\dfrac{H}{H'}\dim_{\rho_{H}}(F)\vee\left(\dim_{\rho_{H}}(F)+\dfrac{1}{H'}-\dfrac{1}{H}\right)\leq\dim_{\rho_{H'}}(F)\leq\left(\dim_{\rho_{H}}(F)\wedge\left(d+\dfrac{H}{H'}\left(\dim_{\rho_{H}}(F)-d\right)\right)\right).
\]
\end{proposition}
\begin{proof}
	Firstly let us remark that for the infimum in \eqref{H-parabolic Hausdorff content} does not change if we consider only $\delta_j\leq 1$.  Therefore an immediate consequence of the definition is
	\begin{equation}\label{Lower-Est-1}
	\dim_{\Psi,H}(F)\leq \dim_{\Psi, H'}(F).
	\end{equation}
	Now let $0<\varepsilon<1$ and $\gamma>\dfrac{\dim_{\Psi,H'}(F)-1}{H'}$.
	Then $\Psi_{H'}^{\gamma H'+1}(F)=0$, and hence there exists a cover
	$\left([a_{n},a_{n}+\delta_{n}]\times\underset{j=1}{\overset{d}{\prod}}[b_{n,j},b_{n,j}+\delta_{n}^{H'}]\right)_{n\geq1}$of
	the set $F$, such that 
	\begin{equation}
	\ensuremath{{\displaystyle \underset{n\geq1}{\sum}\delta_{n}^{\gamma H'+1}\leq\varepsilon.}}\label{eq: Haus-H'}
	\end{equation}
	It follows from $\eqref{eq: Haus-H'}$that $\delta_{n}<1$ for all $n$. Each interval $[a_{n},a_{n}+\delta_{n}]$ can be divided into
	$\left\lceil \delta_{n}^{1-\frac{H'}{H}}\right\rceil $ intervals
	of length $\delta_{n}^{H'/H}$ each. In this way we obtain a new cover
	$\left(\left[a'_{l},a'_{l}+\delta_{l}^{H'/H}\right]\times\underset{j=1}{\overset{d}{\prod}}\left[b'_{l,j},b'_{l,j}+\left(\delta_{l}^{H'/H}\right)^{H}\right]\right)_{l\geq1}$
	of the set $F$ which satisfies
	\begin{equation}
	\ensuremath{{\displaystyle \underset{l\geq1}{\sum}\,\left(\delta_{l}^{H'/H}\right)^{\gamma H+1}\leq2\underset{n\geq1}{\sum}\,\delta_{n}^{1-\frac{H'}{H}}\,\left(\delta_{n}^{H'/H}\right)^{\gamma H+1}\leq2\varepsilon.}}\label{eq:Haus-H}
	\end{equation}
	From $\eqref{eq:Haus-H}$ we deduce that $\dim_{\Psi,H}(F)\leq\gamma H+1$,
	which implies $\dfrac{\dim_{\Psi,H}(F)-1}{H}\leq\gamma$. Therefore
	letting $\gamma$ go to $\dfrac{\dim_{\Psi,H'}(F)-1}{H'}$ we conclude
	\begin{equation}\label{Lower-Est-2}
	\dfrac{H'}{H}\dim_{\Psi,H}(F)+1-\dfrac{H'}{H}\leq\dim_{\Psi,H'}(F).
\end{equation}
	Combining \eqref{Lower-Est-1} and \eqref{Lower-Est-2} , we obtain the lower bound 
	$$\dim_{\Psi,H}(F)\vee 	\left( \dfrac{H'}{H}\dim_{\Psi,H}(F)+1-\dfrac{H'}{H}\right) \leq \dim_{\Psi, H'}(F) .$$
	For the second inequality let $\kappa<\dim_{\Psi,H'}(F)$. Then by Frostman's theorem \ref{Frostman parabolic theorem} there exists a probability measure $\mu$ supported on $F$ such that \eqref{ineq Frostman parabolic} is satisfied. Our aim is to show that
	\begin{align}
	\mu\left(\left[a, a+\delta\right] \times\prod_{j=1}^d\left[b_{j}, b_{j}+\delta^{H}\right]\right)\leq C \left(\delta^{\kappa+d(H-H')}\wedge \delta^{\kappa H/H'}\right), \label{principe}
	\end{align}
	for any $a\in \mathbb{R}_+$, $b_1,...,b_d\in \mathbb{R}$ and $0<\delta<1$. Note first that since $\delta \leq \delta^{H/H'}$ we have
	$$\left[a,a+\delta\right]\times\prod_{j=1}^{d}\left[b_{j},b_{j}+\delta^{H}\right]\subset\left[a,a+\delta^{H/H'}\right]\times\left[b_{1},b_{1}+\left(\delta^{H/H'}\right)^{H'}\right]\times\ldots\times\left[b_{d},b_{d}+\left(\delta^{H/H'}\right)^{H'}\right].$$
	Using \eqref{ineq Frostman parabolic} we obtain 
	\begin{align}
	\mu\left(\left[a, a+\delta\right] \times\prod_{j=1}^d\left[b_{j}, b_{j}+\delta^{H}\right]\right)\leq C \delta^{\kappa\, H/H'}.
	\end{align}
	Now, for any $j=1,...,d$, the interval $[b_j,b_j+ \delta^{H}]$ can be covered by at most $\delta^{(H-H')}$ interval of length $\delta^{H'}$. Using once again \eqref{ineq Frostman parabolic} we deduce that 
	\begin{align}
	\mu\left(\left[a, a+\delta\right] \times\prod_{j=1}^d\left[b_{j}, b_{j}+\delta^{H}\right]\right)\leq C \delta^{\kappa+d(H-H')}.\label{Frostman for K covering}
	\end{align} 
	Thus the inequality \eqref{principe} is proved. Since, in the metric space $(\mathbb{R}_+\times\mathbb{R}^d,\rho_H)$, the diameter of the set $\left[a, a+\delta\right] \times\prod_{j=1}^d\left[b_{j}, b_{j}+\delta^{H}\right]$ is $\delta^{H}$ then the Mass Distribution Principle, (see Theorem $4.19$ in \cite{MP}),
implies that 
$$\dim_{\rho_{H}}(F)\geq\dfrac{\kappa}{H'}\vee\dfrac{\kappa+d(H-H')}{H}.$$
 From \eqref{dim-comp} it follows that $$\kappa\leq\left(\frac{H'}{H}\dim_{\Psi,H}(F)\right)\wedge\left(\dim_{\Psi,H}(F)+d(H'-H)\right).$$
Therefore letting $\kappa\uparrow \dim_{\psi,H'}(F)$ the assertion \eqref{comparison dim_H dim_K} follows. 
\end{proof}
\begin{remark}
It is easy to see that
\[
\dim_{\Psi,H}(F)\vee\left(\dfrac{H'}{H}\dim_{\Psi,H}(F)+1-\dfrac{H'}{H}\right)=\left\{ \begin{array}{llc}
\dim_{\Psi,H}(F) & \text{if} & \dim_{\Psi,H}(F)\leq1\\
\\
\dfrac{H'}{H}\dim_{\Psi,H}(F)+1-\dfrac{H'}{H} & \text{if} & \dim_{\Psi,H}(F)>1.
\end{array}\right.
\]
Hence the lower bound in \eqref{comparison dim_H dim_K} is sharp as long as $\dim_{\Psi,H}(F)>1$. 
\end{remark}
The following proposition looks at the effect of Hölder continuous maps on the Hausdorff dimension of sets $\dim(A)$ and $\dim_{\Psi,H}(Gr_A(f))$ 
\begin{proposition}\label{lemma equality dimensions}
	Let $f$ : $[0,1]\rightarrow \mathbb{R}^d$ be a Borel measurable function and $A$ be a Borel subset of $[0,1]$. Then we have  
	\begin{equation}
	\dim (A)\leq \dim_{\Psi, H}(Gr_A(f)).\label{dim<dim_H}
	\end{equation}
	If in addition the function $f$ is Hölder continuous with exponent $\alpha\leq H$ ($\alpha$-Hölder continuous), that is 
	\[
	\exists\,K\geq0:\lVert f(x)-f(y) \rVert_{\infty}\leq K\vert x-y\vert^{\alpha},\,\,\,\forall x,y\in\left[0,1\right],
	\]
	then we have
	\begin{equation}
	\dim_{\Psi, H}(Gr_A(f)) \leq \left( \frac{H}{\alpha}\dim(A)\right) \wedge \left( \dim(A)+(H-\alpha)d\right) .\label{comparison}
	\end{equation}	 
	Especially, when $f$ is $\left( H-\varepsilon\right) $-Hölder continuous for all $\varepsilon>0$ then 
	\begin{equation}\label{Haus-Lip}
		\dim (A)= \dim_{\Psi, H}(Gr_A(f))
	\end{equation}
\end{proposition}
\begin{proof}
We begin by proving \eqref{dim<dim_H}. Let $\gamma>\dim_{\Psi, H}(Gr_A(f))$. Since $\Psi_{H}^{\gamma}(Gr_A(f))=0$ then for all $\varepsilon>0$ there exists a cover $\left([a_l,a_l+\delta_l]\times \underset{j=1}{\overset{d}{\prod}}[b_{l,j},b_{l,j}+\delta_l^H]\right) _{l\geq 1}$ of $Gr_A(f)$ such that 
$\displaystyle\underset{l\geq 1}{\sum}\delta_l^{\gamma}\leq \varepsilon$. Consequently, in $\mathbb{R}$ with the absolute-value metric,  $([a_l,a_l+\delta_l])_{l\geq 1}$ is a covering of $A$  such that $$\sum_{l\geq1}\left|[a_{l},a_{l}+\delta_{l}]\right|^{\gamma}\leq\varepsilon.$$ 
Here $\left|[a_{l},a_{l}+\delta_{l}]\right|$ is the diameter of the set $[a_l,a_l+\delta_l]$.  Then we deduce that  the $\gamma$-dimensional Hausdorff content satisfies
\begin{align}
\mathcal{H}^{\gamma}(A):=\inf \left\{\sum_{j} \lvert U_{j}\rvert^{\gamma} : A \subseteq \underset{j}{\cup} \, U_{j} \right\}\leq \varepsilon.
\end{align} Thus $\mathcal{H}^{\gamma}(A)=0$  and therefore $\dim(A)\leq \gamma$. Now by making $\gamma \downarrow \dim_{\Psi, H}(Gr_A(f))$ we get the desired inequality. 

Now let $\alpha\leq H$ and assume that $f$ is $\alpha$-Hölder continuous. For any $0<\kappa<\dim_{\Psi, H}(Gr_A(f))$, by resorting again to Frostman's theorem there exist a measure $\mu$ on $Gr_A(f)$ and a constant $C>0$ such that 
\begin{equation}\label{ineq Frostman parabolic-1}
	\mu\left([a, a+\delta] \times \prod_{j=1}^d\left[b_{j}, b_{j}+\delta^{H}\right]\right) \leq C \delta^{\kappa},
\end{equation} 
for any $\left(a,b_{1},\cdots,b_{d} \right)\in  A\times\mathbb{R}^d$ and $\delta \in \left(0,1 \right] $.
Let $\nu$ be the measure on $A$ satisfying $\nu=\mu \circ P_1^{-1}$ where $P_1$ is the projection mapping on $A$, i.e. $P_1(s,f(s))=s$. Our aim is to show that   
\begin{equation}\label{principe-1}
\nu([a,a+\delta])=\mu(P_{1}^{-1}([a,a+\delta]))\leq C_{1}\left(\delta^{\kappa+d(\alpha-H)}\,\wedge \,\delta^{\kappa\,\alpha/H}\right),
\end{equation}
for some constant $C_{1}>0$ and any $\delta \in \left(0,1 \right]$.
Since $f$ is Hölder continuous function with exponent $\alpha$ and constant $K$, then for $s \in [a,a+\delta]$ we have $f_j(s)\in [f_j(a)-K \delta^{\alpha},f_j(a)+K \delta^{\alpha}]$. There are two ways to cover
$$P_1^{-1}\left( [a,a+\delta]\right) =\left\lbrace (s,f(s))\in Gr_A(f) ; s\in [a,a+\delta] \right\rbrace .$$
The first one is :
for all $j=1,...,d$ we decompose every interval $[f_j(a)-K \delta^{\alpha},f_j(a)+K \delta^{\alpha}]$ into at most $\lceil 2K\rceil \delta^{(\alpha-H)}$ interval of length $\delta^H$, then $P_1^{-1}([a,a+\delta])$ is covered by at most $\lceil 2K\rceil^{d}\delta^{d(\alpha-H)}$ sets of the form
$$ [a,a+\delta]\times\prod_{j=1}^{d}[b_j,b_j+\delta^H].$$
With regard to the second way, since $\delta \leq \delta^{\alpha/H}$ then $P_1^{-1}([a,a+\delta])$ may be covered by at most $\lceil 2K\rceil^{d}$ sets of the form $$[a,a+\delta^{\alpha/H}]\times \prod_{j=1}^{d}[b_j,b_j+\delta^{\alpha}].$$
Now using $\eqref{ineq Frostman parabolic-1}$ we conclude that there exists a constant $C'>0$ which depends only on $K$ and $d$, such that $$ \nu([a,a+\delta])=\sigma(P_1^{-1}([a,a+\delta]))\leq C' \delta ^{\kappa+d(\alpha-H)},$$ for the first cover. For the second one we get
$$ \nu([a,a+\delta])\leq  C'' \delta^{\kappa\,\alpha/H},$$ 
where  $C''>0$ is a constant depending only on $K$ and $d$. Thus inequality \eqref{principe-1} is proved. Thanks to the mass distribution principle, see Theorem $4.19$ in \cite{MP}, we obtain 
$$\kappa \leq \left( \frac{H}{\alpha}\dim(A)\right) \wedge \left( \dim(A)+(H-\alpha)d\right).$$
Therefore letting $\kappa \uparrow \dim_{\Psi, H}(Gr_A(f))$ the assertion \eqref{comparison} follows.  

Finally it is easy to see from \eqref{dim<dim_H} and  \eqref{comparison} that when $f$ is $\left( H-\varepsilon\right) $-Hölder continuous for all $\varepsilon>0$ we have $ 	\dim (A)= \dim_{\Psi, H}(Gr_A(f))$.
\end{proof}
\begin{remark}
It is worth noting that the inequality \eqref{comparison} is also obtained by assuming only that the function $f$ is $(\alpha-\varepsilon)$-Hölder continuous for all $\varepsilon>0$.
\end{remark}
A natural question that arises from Proposition \ref{lemma equality dimensions} whether \eqref{comparison} is a sharp estimate?
The main idea is to use the paths of Hölder continuous stochastic process. Precisely, we have the following result
\begin{theorem}\label{parabolic dim of sample path FBM}
	Let $\alpha\leq H$, $\{B^{\alpha}(t): t \in [0,1]\}$  a $d$-dimensional fractional Brownian motion of Hurst index $\alpha$ and  $A\subset [0,1]$ a Borel set. Then we have
	\begin{align}
	\dim_{\Psi,H}(Gr_A(B^{\alpha}))=\left(\left( \frac{H}{\alpha}\dim(A)\right) \wedge(\dim(A)+d(H-\alpha))\right) \text{ a.s.}\label{parabolic dimension of FBM}
	\end{align}
Moreover, if $\alpha<H$ we have 
\begin{equation}
\dim_{\Psi,H}(Gr_A(B^{\alpha}))>\dim(A).\label{strict ineq}
\end{equation}
\end{theorem}

For the proof we need the following lemma.
\begin{lemma}\label{estimation kernel}
	There exists a constants $C$ such that for all $t \in (0,1]$ we have 
\begin{equation}\label{cap-estimate}
\mathbb{E}\left[\frac{1}{(\max\{t^{H},\lVert B^{\alpha}(t)\rVert_{\infty}\})^{\gamma/H}}\right]\leq\left\{ \begin{array}{lll}
C\,t^{-\gamma\alpha/H}\mbox{ } & \text{if} & \gamma<Hd,\\
\\
C\, t^{d(H-\alpha)-\gamma}\mbox{  } & \text{if} & \gamma>Hd.
\end{array}\right.
\end{equation}
\end{lemma}
\begin{proof}
We first note that the denominator on the left hand side of \eqref{cap-estimate} has the same distribution
as $(\max\{t^{H},t^{\alpha}\lVert N\rVert_{\infty}\})^{\gamma/H}$,
where $N$ is a $d$-dimensional standard normal random variable.

Now assume that $\gamma<Hd$. Then we obtain
\[
\mathbb{E}\left[\frac{1}{\left(\max\left(t^{H},\lVert B^{\alpha}(t)\rVert_{\infty}\right)\right){}^{\gamma/H}}\right]=\mathbb{E}\left[\frac{1}{\left(\max\left(t^{H},t^{\alpha}\lVert N\rVert_{\infty}\right)\right)^{\gamma/H}}\right]\leq t^{-\alpha\gamma/H}\mathbb{E}\left[\lVert N\rVert_{\infty}^{-\gamma/H}\right].
\]

Since $\gamma<Hd$, we have that $\mathbb{E}\left[\lVert N\rVert_{\infty}^{-\gamma/H}\right]$
is finite and therefore 
\[
\mathbb{E}\left[\frac{1}{\left(\max\left(t^{H},\lVert B^{\alpha}(t)\rVert_{\infty}\right)\right){}^{\gamma/H}}\right]\leq C\,t^{-\alpha\gamma/H}.
\]
Next, let $\gamma>Hd$. Since $\lVert y\rVert_{\infty}\leq \lVert y\rVert$ for any $ y\in \mathbb{R}^d$  it yields
\[
\begin{array}{lll}
\mathbb{E}\left[\dfrac{1}{\left(\max\left(t^{H},t^{\alpha}\lVert N\rVert_{\infty}\right)\right)^{\gamma/H}}\right] & = & \dfrac{1}{t^{\gamma}}\mathbb{P}\left[\lVert N\rVert_{\infty}\leq t^{H-\alpha}\right]+\dfrac{1}{t^{\gamma\alpha/H}}\mathbb{E}\left[\dfrac{1}{\lVert N\rVert_{\infty}^{\gamma/H}}1\!\!1_{\{\lVert N\rVert_{\infty}>t^{H-\alpha}\}}\right]\\
\\
& \leq & C\left(\dfrac{1}{t^{\gamma}}\dint_{\lVert y\rVert_{\infty}\leq t^{H-\alpha}}\,e^{-\lVert y\rVert^{2}/2}dy+\dfrac{1}{t^{\gamma\alpha/H}}\dint_{\lVert y\rVert_{\infty}>t^{H-\alpha}}\,\frac{e^{-\lVert y\rVert^{2}/2}}{\lVert y\rVert_{\infty}^{\gamma/H}}dy\right)\\
\\
& \leq & C_1\left(t^{d(H-\alpha)-\gamma}+t^{-\gamma\alpha/H}\dint_{t^{H-\alpha}}^{1}r^{d-1-\gamma/H}dr+t^{-\gamma\alpha/H}\right)\\
\\
& \leq & C_2\left(t^{d(H-\alpha)-\gamma}+t^{-\gamma\alpha/H}\right)\\
\\
& \leq & C_3\,t^{d(H-\alpha)-\gamma},
\end{array}
\]
where the constants $C_j, \,j=1,2,3$ depend only on $d$, $H$, $\alpha$ and $\gamma$.
\end{proof}


\begin{proof}[Proof of Theorem \ref{parabolic dim of sample path FBM}]
The upper bound of $\dim_{\Psi,H}(Gr_A(B^{\alpha}))$ follows directly from Proposition \ref{lemma equality dimensions}. 

Now let us prove the lower bound. First, we consider the case  $\dim(A)\leq \alpha d$.  Let $ \gamma<\frac{H}{\alpha}\dim(A)\leq \dim(A)+d(H-\alpha)$. By Theorem 4.32 in \cite{MP}, there exists a probability measure $\nu$ on $A$ such that 
\begin{align}
\mathcal{E}_{\gamma\alpha/H}(\nu):=\int_A\int_A\frac{1}{|t-s|^{\gamma\alpha/H}}\nu(ds)\nu(dt)<\infty\label{parabolic dim of FBM: Frostman}.
\end{align} 
Let $\widetilde{\mu}$ be the random measure defined by $$ \widetilde{\mu}(E)=\nu\{s: (s,B^{\alpha}(s))\in E\},$$
where $E\subset Gr_A(B^{\alpha})$. We will show that $$\mathcal{E}_{\rho_{H},\gamma/H}(\widetilde{\mu}):=\int_{\mathbb{R}_+\times\mathbb{R}^{d}} \int_{\mathbb{R}_+\times\mathbb{R}^{d}} \frac{\widetilde{\mu}(d u) \widetilde{\mu}(d v)}{(\rho_H(u,v))^{\gamma/H}}<\infty \text{ a.s.}$$
Taking expectation and using a change of variables and Fubini’s theorem we get 
\[
\mathbb{E}\left[\mathcal{E}_{\rho_{H},\gamma/H}(\widetilde{\mu})\right]=\int_{A}\int_{A}\mathbb{E}\left[\frac{1}{\left(\max\left(|s-t|^{H},\lVert B^{\alpha}(t)-B^{\alpha}(s)\rVert_{\infty}\right)\right){}^{\gamma/H}}\right]\nu(ds)\nu(dt).
\]
By recalling the fact that $B^{\alpha}$ has stationary increments we see that the denominator has the same distribution as $\left(\max\left(|s-t|^{H},\lVert B^{\alpha}(\vert s-t\vert)\rVert_{\infty}\right)\right){}^{\gamma/H}$.
It follows that 
\[
\mathbb{E}\left[\mathcal{E}_{\rho_{H},\gamma/H}(\widetilde{\mu})\right]=\int_{A}\int_{A}\mathbb{E}\left[\frac{1}{\left(\max\left(|s-t|^{H},\lVert B^{\alpha}(\vert s-t\vert)\rVert_{\infty}\right)\right){}^{\gamma/H}}\right]\nu(ds)\nu(dt).
\]
Since $\gamma<Hd$ we deduce from Lemma \ref{estimation kernel} that
\[
\mathbb{E}\left[\mathcal{E}_{\rho_H,\gamma/H}(\widetilde{\mu})\right]\leq C\,\int_A\int_A\frac{1}{|t-s|^{\gamma\alpha/H}}\nu(ds)\nu(dt),
\]
which is finite from \eqref{parabolic dim of FBM: Frostman}. Hence $\mathcal{C}_{\rho_H,\gamma/H}(Gr_A(B^{\alpha}))>0$ almost surely  and Proposition \ref{capacity approach to dim_H} allows that $\dim_{\Psi,H}(Gr_A(B^{\alpha}))\geq \gamma\,\,a.s$. 

Now assume that $\dim(A)>\alpha d$. In this case, we choose $Hd<\gamma<\dim(A)+d(H-\alpha)<\frac{H}{\alpha}\dim(A)$. We use the same tools, as in the previous case, to prove that $\mathcal{C}_{\rho_H,\gamma/H}(Gr_A(B^{\alpha}))>0\,\,a.s$ via the probability measure $\nu$ satisfying 
\[
\mathcal{E}_{\gamma-d(H-\alpha)}(\nu)<\infty.\label{parabolic dim of FBM: Frostman 2}
\]
Letting $\gamma\uparrow (\frac{H}{\alpha}\dim(A))\wedge(\dim(A)+d(H-\alpha))$ finishes the proof. 
\end{proof}
As a consequence of \eqref{dim-comp} and Theorem \ref{parabolic dim of sample path FBM}, we have the following result 
\begin{corollary}
	Let $\alpha\leq H$, $\{B^{\alpha}(t):t\in[0,1]\}$ a fractional Brownian
	motion of Hurst index $\alpha$ and $A\subset[0,1]$ a Borel
	set. Then we have
	
	\[
	\dim_{\Psi,H}\left(Gr_{A}(B^{\alpha})\right)>Hd\,\,a.s\Longleftrightarrow\dim_{\rho_{H}}\left(Gr_{A}(B^{\alpha})\right)>d\,\,a.s\Longleftrightarrow\dim(A)>\alpha d.
	\]
\end{corollary}
\section{Positive Lebesgue  measure and non-empty interior of $ (B ^ {H} +f) (A) $}

Let $Y=(Y(t))_{t\in[0,1]}$ be an $\mathbb{R}^{d}$-valued stochastic process and  $\nu$ is a positive
measure on $[0,1]$. The occupation
measure of the sample path $[0,1]\ni t\longrightarrow Y(t)(w)\in\mathbb{R}^{d}$
is defined by 
\[
\nu_{Y}(E):=\nu\left\lbrace t\in[0,1]:Y(t)\in E\right\rbrace ,
\]
where $E\subset\mathbb{R}^{d}$ is a Borel set. We say that $Y$ has an occupation density relative
to the Lebesgue measure $\lambda_{d}$ if $\nu_{Y}$ is absolutely continuous
with respect to $\lambda_{d}$ almost surely.

Simple modifications of the differentiation's method, see Theorem 21.15 in \cite{GH},  allow  to give necessary and sufficient conditions  under which $\nu_{Y} $ has an occupation density relative to the Lebesgue measure $\lambda_d$ a.s. Hence we omit the proof. Here is a precise statement:

\begin{proposition}\label{abs-cont}
	The following assertions are equivalent:
	\begin{equation*}
	\begin{aligned}
	1.\qquad &\nu_{Y} <<\lambda_d  \,\,\,\text{with} \,\,\,  \frac{d\nu_{Y}}{d\lambda_d}(.)\in L^2(\lambda_d \otimes \mathbb{P}).\\ \\
	2.\qquad &  \liminf_{r \downarrow 0}\, r^{-d}\int_{A}\int_{A}\mathbb{P}\left\lbrace \lVert Y(s)-Y(t)\rVert <r\right\rbrace  \,d\nu(s)\,d\nu(t)<\infty. 
	\end{aligned}
	\end{equation*}
\end{proposition} 

Let $A \subset [0,1]$ and assume that $\eta:=\dim_{\Psi, H}(Gr_A(f))>Hd$. It follows from Theorem \ref{Frostman parabolic theorem} that for $\kappa \in (Hd,\eta)$ there exists a Borel probability measure $\mu$ supported on $Gr_A(f)$ and $C>0$ such that
\begin{equation}\label{ineq Frostman parabolic-2}
\mu\left([a, a+\delta] \times \prod_{j=1}^d\left[b_{j}, b_{j}+\delta^{H}\right]\right) \leq C \delta^{\kappa},
\end{equation}
for all $a \in A$, $b_1,...,b_n \in \mathbb{R}^d$ and all $\delta\in \left( 0,1\right] $. 
Let $\nu$ be the measure defined on $A$ by $\nu=\mu \circ P_1^{-1}$. We are now ready to state
\begin{theorem}\label{main th 1}
	Let $\{B^{H}(t): t\in [0,1]\}$ be a $d$-dimensional fractional Brownian motion of Hurst index $H\in (0,1)$. Let $f:[0,1]\rightarrow \mathbb{R}^d$ be a Borel measurable function and let $A\subset [0,1]$ be a Borel set. If $\eta:=\dim_{\Psi, H}(Gr_A(f))>Hd$ then $\lambda_d((B^{H}+f)(A))>0$ almost surely.
\end{theorem}

\begin{proof}
To achieve this purpose, it is enough to verify the second assertion of Proposition \ref{abs-cont} for the process $B^{H}+f$. Indeed for $s,t \in A$ and $r>0$, we have

\begin{align*}
\mathbb{P}\left\lbrace \lVert(B^{H}+f)(s)-(B^{H}+f)(t)\rVert<r\right\rbrace =\dfrac{1}{(2\pi)^{d/2}|t-s|^{Hd}}\int_{\lVert y\rVert<r}\exp\left(-\frac{\lVert y-f(t)+f(s)\rVert^{2}}{2|t-s|^{2H}}\right)dy.
\end{align*} 
Then using Fubini's theorem we obtain, for any fixed $t\in A$ and $r>0$, that 
\begin{align}
\dint_{A}\mathbb{P}\left\lbrace \lVert(B^{H}+f)(s)-(B^{H}+f)(t)\rVert<r\right\rbrace  \,d\nu(s) & = \dfrac{1}{(2\pi)^{d/2}}\int_{\lVert y\rVert<r}\int_{A}\dfrac{1}{|t-s|^{Hd}}\exp\left(-\frac{\lVert y-f(t)+f(s)\rVert^{2}}{2|t-s|^{2H}}\right)\,d\nu(s)\,dy
\nonumber\\\nonumber
\\
& \leq C\,r^{d}\sup_{\lVert y\rVert<r}\underset{=I(y,\left( t,f(t))\right) }{\underbrace{\int_{A}\,\dfrac{1}{|t-s|^{Hd}}\exp\left(-\frac{\lVert y-f(t)+f(s)\rVert^{2}}{2|t-s|^{2H}}\right)\,d\nu(s)}}.
\label{estim proba by sup}
\end{align}
where $C$ is a positive constant depending only on $d$.

For any fixed $y$ in $\left\lbrace \lVert y\rVert<r \right\rbrace $ and $ \left( t,f(t)\right)\in Gr_{A}(f) $, we have the following decomposition

$$ I\left(y,\left(  t,f(t)\right) \right)=I_1\left(y,\left(  t,f(t)\right) \right)+I_2\left(y,\left(  t,f(t)\right) \right),$$
where
\[
\begin{array}{l}
I_{1}\left(y,\left(t,f(t)\right)\right) = \dint_{\left\lbrace s\in A:\lVert y-f(t)+f(s)\rVert\leq C_{1}\,\vert s-t\vert^{H}\sqrt{\vert\log\vert s-t\vert\vert}\right\rbrace }\dfrac{1}{|s-t|^{Hd}}\exp\left(-\frac{\lVert y-f(t)+f(s)\rVert^{2}}{2|s-t|^{2H}}\right)\,d\nu(s)\\
\\
= \dint_{\left\lbrace (s,f(s))\in Gr_{A}(f):\lVert y-f(t)+f(s)\rVert\leq C_{1}\,\vert s-t\vert^{H}\sqrt{\vert\log\vert s-t\vert\vert}\right\rbrace }\dfrac{1}{|s-t|^{Hd}}\exp\left(-\frac{\lVert y-f(t)+f(s)\rVert^{2}}{2|s-t|^{2H}}\right)\,d\mu(s,f(s)).
\end{array}\]
and 
\[
\begin{array}{l}
I_{2}\left(y,\left(t,f(t)\right)\right) =\dint_{\left\lbrace s\in A:\lVert y-f(t)+f(s)\rVert>C_{1}\,\vert s-t\vert^{H}\sqrt{\vert\log\vert s-t\vert\vert}\right\rbrace }\dfrac{1}{|s-t|^{Hd}}\exp\left(-\frac{\lVert y-f(t)+f(s)\rVert^{2}}{2|s-t|^{2H}}\right)\,d\nu(s)\\
\\
= \dint_{\left\lbrace (s,f(s))\in Gr_{A}(f):\lVert y-f(t)+f(s)\rVert>C_{1}\,\vert s-t\vert^{H}\sqrt{\vert\log\vert s-t\vert\vert}\right\rbrace }\dfrac{1}{|s-t|^{Hd}}\exp\left(-\frac{\lVert y-f(t)+f(s)\rVert^{2}}{2|s-t|^{2H}}\right)\,d\mu(s,f(s)),
\end{array}
\]
where $C_{1}$ is a positive constant which will be chosen later.
We first show that $$\underset{\left( y,\left(t,f(t)\right)\right) \in \left\lbrace \lVert y\rVert<r \right\rbrace \times Gr_{A}(f)}{\sup}I_{1}\left(y,\left(  t,f(t)\right) \right)<+\infty.$$ 
By the inequality \eqref{ineq Frostman parabolic-2} we can verify that the measure $\nu$ is no atomic. Then we have

\begin{align*}
I_1(y,\left( t,f(t))\right)&\leq \sum_{n=1}^{\infty}\,2^{nHd}\,\mu \left\lbrace \lVert y-f(t)+f(s)\rVert  \leq C_1\,\vert s-t\vert^{H}\sqrt{\vert\log\vert s-t\vert\vert},\,\,  2^{-n}<|s-t|\leq 2^{1-n}\right\rbrace\\
&\leq \sum_{n=1}^{\infty}\,2^{nHd}\,\mu \left\lbrace \lVert y-f(t)+f(s)\rVert \leq C_2 \,2^{-nH}\sqrt{n},\,\,  2^{-n}<|s-t|\leq 2^{1-n}\right\rbrace,
\end{align*}
where $C_2=C _1\,2^H\sqrt{\log(2)}$.
Now, for  $n\geq1$, we set $$S_n\left(y,\left( t,f(t)\right)\right)  =\left\lbrace  \left( s,f(s)\right)\in Gr_{A}(f):\lVert y-f(t)+f(s)\rVert \leq C_2 \,2^{-nH}\sqrt{n},\,\,  2^{-n}<|s-t|\leq 2^{1-n}\right\rbrace .$$
For  $(s,f(s))\in S_n\left(y,\left( t,f(t)\right)\right) $  we have  $s\in [t-2^{1-n},t+2^{1-n}] $ and  $$f_j(s)\in J_j=\left[ f_j(t)-y_j-C_2 2^{-nH}\sqrt{n},f_j(t)-y_j+C_2 2^{-nH}\sqrt{n}\right]$$ for all $j\in \{1,...,d\}$. Since $\left[ t-2^{1-n},t+2^{1-n}\right] $ is covered by $4$ intervals of length $2^{-n}$ and every $J_j$ is covered by no more than a constant multiple (which depends only on $H$) of $\sqrt{n}$ intervals of length $2^{-nH}$, then we can cover $ S_n\left(y,\left( t,f(t)\right)\right) $ by no more than a constant (which depends only on $d$ and $H$) multiple of $n^{d/2}$ sets of the form  $\left[ a,a+2^{-n}\right] \times\prod_{j=1}^{d}\left[ b_j,b_j+2^{-nH}\right] $. Applying inequality \eqref{ineq Frostman parabolic-2} allows to have 
$$\underset{\left( y,\left(t,f(t)\right)\right) \in \left\lbrace \lVert y\rVert<r \right\rbrace \times Gr_{A}(f)}{\sup}\,\mu( S_n \left(y,\left( t,f(t)\right)\right) \leq C_3 \, n^{d/2} \, 2^{-\kappa n}.$$
Therefore we get
\begin{equation}\label{est 1}
\underset{\left( y,\left(t,f(t)\right)\right) \in \left\lbrace \lVert y\rVert<r \right\rbrace \times Gr_{A}(f)}{\sup}I_{1}\left(y,\left(  t,f(t)\right) \right)\leq C_4 \sum_{n=1}^{\infty}\, 2^{-(\kappa-Hd)n}\, n^{d/2}<\infty,
\end{equation}
where $C_4$ depends on $d$ and $H$ only. For the second term $I_2\left(y,\left(  t,f(t)\right) \right)$ we have 
\begin{align*}
I_2\left(y,\left(  t,f(t)\right) \right)\leq & \sum_{n=1}^{\infty}2^{nHd}\exp\left(- \dfrac{C_1^2}{2}(n-1)\ln2\right) \times\\ &\,\mu \left\lbrace \lVert f(s)-f(t)+y\rVert  > C_1\,\vert s-t\vert^{H}\sqrt{\vert\log\vert s-t\vert\vert},\,\,  2^{-n}<|s-t|\leq 2^{1-n}\right\rbrace
\end{align*}
Thus,  for $C_1>\sqrt{2Hd}$ we obtain
\begin{equation}
\underset{\left( y,\left(t,f(t)\right)\right) \in \left\lbrace \lVert y\rVert<r \right\rbrace \times Gr_{A}(f)}{\sup}I_{2}\left(y,\left(  t,f(t)\right) \right)\leq e^{C_1^2\ln2/2} \sum_{n=1}^{\infty}2^{-n\left( C_1^2/2-Hd\right) }<+\infty.\label{est 2}
\end{equation}
Now putting all this together yields  $$\underset{\left( y,\left(t,f(t)\right)\right) \in \left\lbrace \lVert y\rVert<r \right\rbrace \times Gr_{A}(f)}{\sup}\dint_{A}\,\dfrac{1}{|s-t|^{Hd}}\exp\left(-\frac{\lVert y-f(t)+f(s)\rVert^{2}}{2|s-t|^{2H}}\right)\,d\nu(s)<\infty.$$ 
Thus we get from \eqref{estim proba by sup} that $$\liminf_{r \downarrow 0}r^{-d}\int_{A}\int_{A}\, \mathbb{P}\left\lbrace \lVert(B^{H}+f)(s)-(B^{H}+f)(t)\rVert<r\right\rbrace\,d\nu(s)\,d\nu(t)<\infty.$$
We can therefore state from Proposition \ref{abs-cont} that  the occupation measure
$\nu_{B^{H}+f}$ is absolutely continuous with respect to the Lebesgue
measure $\lambda_{d}$ a.s. Hence $\lambda_{d}\left(B^{H}+f\right)(A)>0$
a.s.
\end{proof}

\begin{remark}
In connection with the existence of the occupation density let us mention that another, often easier to apply, criterion due to Berman tells us that, 
$\nu_{Y}$ has a density $\frac{d\nu_{Y}}{d\lambda_d}(.)\in L^2(\lambda_d \otimes \mathbb{P})$ if and only if
\begin{equation}\label{Berman}
\int_{\mathbb{R}^{d}}\int_{A}\int_{A}\,\mathbb{E}\left(e^{i\langle\theta,(Y(s)-Y(t)\rangle}\right)\,d\nu(s)\,d\nu(t)\,d\theta<\infty.
\end{equation}
	However, we were unable to apply it to prove the absolute
	continuity of the occupation measure $\nu_{B^{H}+f}$ with respect
	to the Lebesgue measure $\lambda_{d}$ a.s. The difficulty stems from
	the lack of informations on the measure $\nu$. Indeed, a simple calculation using characteristic function of Gaussian random
	vector yields
	\[
	\mathbb{E}\left(e^{i\langle\theta,(B^{H}+f)(s)-(B^{H}+f)(t)\rangle}\right)=e^{i\langle\theta,f(s)-f(t)\rangle}\exp\left(-\frac{|s-t|^{2H}\lVert\theta\rVert^{2}}{2}\right).
	\]
	Now integrating the modulus we have 
	\[
	\int_{\mathbb{R}^{d}}\,\left|\mathbb{E}\left(e^{i\langle\theta,(B^{H}+f)(s)-(B^{H}+f)(t)\rangle}\right)\right|\,d\theta=\frac{\left(2\pi\right)^{d/2}}{|s-t|^{Hd}}.
	\]
	However we are unable to ensure finiteness of the integral, 
	\[
	\int_{A}\int_{A}\,\frac{\left(2\pi\right)^{d/2}}{|s-t|^{Hd}}\,d\nu(s)\,d\nu(t).
	\]
	Hence one cannot use Fubini's theorem which make the cited criterion unworkable.
\end{remark}

According to the conclusion of Theorem \ref{main th 1}, it may be expected that the interior of $(B^{H}+f)(A)$ is not empty. Recall that Kahane (Theorems 1 and 2 p. 267 in \cite{Kahane}) has shown that, for any compact subset $A\subset [0,1]$, $B^{H}(A)$ is a Salem set almost surely if $\dim (A)<Hd$ and $B^{H} (A)$ has a non-empty interior almost surely if $\dim (A) > Hd$. That's what prompted us to consider the case $Hd<\dim(A)$. Let us note that in the light of \eqref{dim<dim_H}, the condition $Hd<\dim(A)$ leads to $\dim_{\Psi, H}(Gr_A(f))>Hd$ and therefore $\lambda_d((B^{H}+f)(A))>0$ almost surely.  
Our aim is to prove that, under this condition, the set $(B^{H}+f)(A)$ not only has a positive measure but also a non-void interior. A key ingredient  is the continuity of the occupation density over $A$.  Obtaining this continuity we will make use of the following.

It follows from Lemma 7.1 of \cite{Pi} that, for any $H\in(0,1)$,
the real-valued fractional Brownian motion $B_{0}^{H}$ has the following
important property of strong local nondeterminism: There exists a
constant $0<C<\infty$ such that for all integers $n\ge1$ and all
$t_{1},\ldots,t_{n},t\in[0,1]$, we have
\[
\text{Var}\left(B_{0}^{H}(t)/B_{0}^{H}(t_{1}),\ldots,B_{0}^{H}(t_{n})\right)\geq\,C\,\underset{0\leq j\leq n}{\min}\,\lvert t-t_{j}\rvert^{2H},
\]
where $\text{Var}\left(B_{0}^{H}(t)/B_{0}^{H}(t_{1}),\ldots,B_{0}^{H}(t_{n})\right)$
denotes the conditional variance of $B_{0}^{H}(t)$ given $B_{0}^{H}(t_{1}),\ldots,B_{0}^{H}(t_{n})$
and  $t_{0}=0$. The following elementary
formula allows to estimate the determinant of the covariance
matrix of the Gaussian random vector $(Z_{1},\ldots,Z_{n})$

\begin{equation*}
\text{det}\text{Cov}(Z_{1},\ldots,Z_{n})=\text{Var}(Z_{1})\underset{k=2}{\overset{n}{\prod}}\text{Var}\left(Z_{k}/Z_{1},\ldots,Z_{k-1}\right).
\end{equation*}
For the vector $\left(B_{0}^{H}(t_{1}),\ldots,B_{0}^{H}(t_{n})\right)$
one obtains
\begin{equation}\label{cov-estim}
\text{det}\text{Cov}\left(B_{0}^{H}(t_{j}),1\leq j\leq p\right)\geq\underset{j=1}{\overset{p}{{\displaystyle \prod}}}\left(\min\,\left\{ \lvert t_{j}-t_{i}\rvert^{2H},0\leq i\leq j-1\right\} \right).
\end{equation}

We need also the following lemma which is due to Cuzick and DuPreez \cite{CD} 
\begin{lemma}\label{Cuzick and Du Peez }
Let $Z_{1},...,Z_{n}$ be linearly-independent centered Gaussian variables. If $g:\mathbb{R}\longrightarrow\mathbb{R}$ is a Borel measurable function such that 
	\[
	\int_{-\infty}^{\infty}g(v)e^{-\varepsilon v^{2}}dv<\infty
	\]
	for all $\varepsilon>0$. Then 
	\begin{align*}
	\int_{\mathbb{R}^{n}}g(v_{1}) & \exp\left(-\frac{1}{2}\text{Var}\left(\sum_{j=1}^{n}v_{j}Z_{j}\right)\right)dv_{1}...dv_{n}\\
	& =\frac{(2\pi)^{(n-1)/2}}{(\text{det}\text{Cov}(Z_{1},...,Z_{n}))^{1/2}}\int_{-\infty}^{\infty}g\left(\frac{v}{\sigma_{1}}\right)e^{-v^{2}}dv,
	\end{align*}
	where $\sigma_{1}^{2}=\text{Var}(Z_{1}|Z_{2}...Z_{n})$ is the conditional
	variance of $Z_{1}$ given $Z_{2},...,Z_{n}$. 
\end{lemma}

The next result, using the above-mentioned preliminaries, strengthens the result of Theorem \ref{main th 1} since it shows that $(B^{H}+f)(A)$  has a non-empty interior.

\begin{theorem}\label{main th 2} 
Let $\{B^{H}(t): t\in [0,1]\}$ be a $d$-dimensional fractional Brownian motion of Hurst index $H\in (0,1)$. Let $f:[0,1]\rightarrow \mathbb{R}^d$ be a continuous function and let $A\subset [0,1]$ be a closed set.	Assume that $\dim(A)>Hd$. Then we have $(B^{H}+f)(A)$ has a non-empty interior almost surely.
\end{theorem} 
\begin{proof}
	The usual Frostman's theorem ensure that for all $Hd<\kappa <\dim(A)$ there is a constant $C>0$ and a Borel probability measure $\nu$ on $A$ such that 
	\begin{equation}
	\tau([a,a+\delta])\leq C\, \delta^{\kappa},\label{ineq Frostman usual}
	\end{equation}for all $a\in A$ and $0<\delta<1$. it is easily verified that $$\dint_{\mathbb{R}^{d}}\,\dint_{A}\dint_{A}\,\mathbb{E}\left(e^{i\langle\theta,(B^{H}+f)(s)-(B^{H}+f)(t)\rangle}\right)\,d\tau(s)\,d\tau(t)\,d\theta\leq (2\pi)^{d/2}\,\dint_{A}\dint_{A}\dfrac{d\tau(s)d\tau(t)}{|t-s|^{Hd}},$$
	which is finite by \ref{ineq Frostman usual} since $\kappa >Hd$. Hence the condition \ref{Berman} for the process $B^{H}+f$ and the probability measure $\tau$ is checked. Consequently, the occupation measure $\tau_{B^{H}+f}$ has an occupation density relative to the Lebesgue measure $\lambda_d$ a.s denoted by $\vartheta$. Hence, in order to prove our theorem, it is sufficient to prove that $\vartheta$ has a continuous  version.
	
	First of all, let us note that the arguments used to show $(25.7)$ in \cite{GH} can be modified to see that, for all $x,y \in \mathbb{R}^d$ and for all even integers $p\geq 2$, we have
	\begin{equation*}
	\begin{array}{lll}
	\mathbb{E}\left(\vartheta(x)-\vartheta(y)\right)^{p}= & \left(2\pi\right)^{-pd}\dint_{A^{p}}\dint_{\mathbb{R}^{pd}}\underset{j=1}{\overset{p}{\prod}}\left(\exp\left(i\langle x,\xi_{j}\rangle\right)-\exp\left(i\langle y,\xi_{j}\rangle\right)\right)\\
	\\
	& \mathbb{E}\exp\left(i\displaystyle\underset{j=1}{\overset{p}{\sum}}\langle\xi_{j},\left(B^{H}+f\right)(t_{j})\rangle\right)\,d\xi_{1}\ldots d\xi_{p}\,d\tau(t_{1})\ldots d\tau(t_{p}).
	\end{array}
	\end{equation*}
	Let $p\geq 2$ be a fixed even integer and $0<\gamma<1$ whose value will be determined later. Using  the facts that $$|e^{iu}-1|\leq 2^{1-\gamma} \,\,|u|^{\gamma},\, \,\forall\; u\in \mathbb{R}, $$ and $|a+b|^{\gamma}\leq |a|^{\gamma}+|b|^{\gamma}$, we have 
	\[
	\underset{j=1}{\overset{p}{\prod}}\vert\exp\left(i\langle x,\xi_{j}\rangle\right)-\exp\left(i\langle y,\xi_{j}\rangle\right)\vert\leq2^{(1-\gamma)p}\,\,\lVert x-y\rVert^{\gamma p}\,\sum\prod_{j=1}^{p}|\xi_{j,k_{j}}|^{\gamma}
	\]
	where the summation $\sum$ is taken over all $(k_1,...,k_p)\in \{1,...,d\}^{p}$. It follows that
	\[
	\begin{array}{ll}
	\mathbb{E}\left(\vartheta(x)-\vartheta(y)\right)^{p}  \leq & \left(2\pi\right)^{-pd}\,\,2^{(1-\gamma)p}\,\, \lVert x-y\rVert^{\gamma p}\, \displaystyle\sum\,\dint_{A^{p}}\dint_{\mathbb{R}^{pd}}\;\;\prod_{j=1}^{p}|\xi_{j,k_{j}}|^{\gamma}\,\\
	\\
	& \exp\left(-\dfrac{1}{2}\text{Var}\left(\displaystyle\overset{p}{\underset{j=1}{\sum}}\left\langle \xi_{j},B^{H}(t_{j})\right\rangle \right)\right)\,d\xi_{1}\ldots d\xi_{p}\,\,d\tau(t_{1})\ldots d\tau(t_{p}).
	\end{array}
	\]
	Now the generalized Hölder's inequality leads to 
	\[
	\begin{array}{lll}
	\mathbb{E}\left(\vartheta(x)-\vartheta(y)\right)^{p} & \leq & \left(2\pi\right)^{-pd}\,\,2^{(1-\gamma)p}\,\,\lVert x-y\rVert^{\gamma p}\,{\displaystyle \sum\,\dint_{A^{p}}\prod_{j=1}^{p}}\\
	\\
	&  & \left[\dint_{\mathbb{R}^{pd}}\;\;|\xi_{j,k_{j}}|^{p\gamma}\,\exp\left(-\dfrac{1}{2}\text{Var}\left({\displaystyle \overset{p}{\underset{j=1}{\sum}}\left\langle \xi_{j},B^{H}(t_{j})\right\rangle }\right)\right)\,d\xi_{1}\ldots d\xi_{p}\right]^{1/p}\,\,d\tau(t_{1})\ldots d\tau(t_{p}).
	\end{array}
	\]
	Fix a sequence $(k_{1},...,k_{p})\in\left\{ 1,...,d\right\} ^{p}$
	, $0<t_{1}<\ldots<t_{p}$, and consider 
	\[
	J=\prod_{j=1}^{p}\left[\dint_{\mathbb{R}^{pd}}\;\;|\xi_{j,k_{j}}|^{p\gamma}\,\exp\left(-\dfrac{1}{2}\text{Var}\left({\displaystyle \overset{p}{\underset{j=1}{\sum}}\left\langle \xi_{j},B^{H}(t_{j})\right\rangle }\right)\right)\,d\xi_{1}\ldots d\xi_{p}\right]^{1/p}.
	\]
	Note that $\left\{ B^{H}_{l},1\leq l\leq d\right\} $ are independent
	copies of $B^{H}_{0}$. The strong local nondeterminism of the fractional
	Brownian motion $B^{H}_{0}$ and \eqref{cov-estim}, ensure that the random variables $\left\{ B^{H}_{l}(t_{j}),1\leq l\leq d;1\leq j\leq p\right\} $
	are linearly independent. Hence by applying Lemma \ref{Cuzick and Du Peez }, we derive that $J$ is bounded by
	
	\begin{align}
	J  \leq & \,\,\dfrac{(2\pi)^{(pd-1)/2}}{\left[\det\text{Cov}\left(B^{H}_{l}(t_{j}),0\leq l\leq d,1\leq j\leq p\right)\right]^{1/2}}\,\,\dint_{\mathbb{R}}|r|^{p\gamma}e^{-r^{2}/2}dr\,\prod_{j=1}^{p}\,\,\dfrac{1}{\sigma_{j}^{\gamma}} \notag\\
	\notag\\
	\leq &\,\,\dfrac{(dK)^{p}(p!)^{\gamma}}{\left[\det\text{Cov}\left(B^{H}_{0}(t_{j}),1\leq j\leq p\right)\right]^{d/2}}\,\,\underset{j=1}{\overset{p}{\displaystyle\prod}}\,\,\dfrac{1}{\sigma_{j}^{\gamma}} \label{J est}
	\end{align}
	where $\sigma_j^2$ is the conditional variance of $B^{H}_{k_j}(t_j)$ given $B^{H}_{l}(t_i)$ $(l\neq k_j \text{ or } l=k_j, i\neq j)$ and the last inequality follows from Stirling's formula. Combining \eqref{J est} with the well-known facts, see for example the proof of Lemma $2.5$ in \cite{J12}, that 
	\begin{equation}
	\underset{j=1}{\overset{p}{\displaystyle\prod}}\dfrac{1}{\sigma_j^{\gamma}} \leq \frac{K^p}{\left[\det\text{Cov}\left(B^{H}_{0}(t_{j}),1\leq j\leq p\right)\right]^{\gamma}} \leq \dfrac{K^{p}}{\underset{j=1}{\overset{p}{{\displaystyle \prod}}}\left(\min\,\left\{ \lvert t_{j}-t_{i}\rvert^{2H\gamma},0\leq i\leq j-1\right\} \right)},\label{estim prod sigma}
	\end{equation}
we conclude 
	\begin{equation}
	J\leq\dfrac{\left(dK^{2}\right)^{p}(p!)^{\delta}}{\underset{j=1}{\overset{p}{{\displaystyle \prod}}}\left(\min\,\left\{ \lvert t_{j}-t_{i}\rvert ^{H(d+2\gamma)},0\leq i\leq j-1\right\} \right)}.
	\label{estim J(t) 2}
	\end{equation}
	Now, we can prove that
	\begin{equation}
	\dint_{A^p}\,\dfrac{d\tau(t_{1})\ldots d\tau(t_{p})}{\underset{j=1}{\overset{p}{{\displaystyle \prod}}}\left(\min\,\left\{ |t_{j}-t_{i}|,0\leq i\leq j-1\right\} \right)^{H(d+2\gamma)}}\leq  p! \left( \sup_{s\in [0,1]}\dint_{A}\frac{d\tau(t)}{|t-s|^{H(d+2\gamma)}}\right) ^{p}.\label{estim integral}
	\end{equation}
	 First,  for all $j\in \{1,...,p\}$ we splite $A$ into the regions $A_{i,j}:=\left\{t\in A\cap I , |t-t_i|=\min \{|t-t_l|, 0\leq l\leq j-1\}\right\}$ for all $i\in \{0,...,j-1\}$, and $t_0=0$. Then, for all $ j\in \{1,...,p\}$ we have 
	\begin{align*}
	\int_{A}\frac{d\nu(t)}{(\min \{|t-t_i|, 0\leq i\leq j-1\})^{H (d+2\gamma)}}=\sum_{i=0}^{j-1}\int_{A_{i,j}}\frac{d\nu(t)}{|t-t_i|^{H (d+2\gamma)}} \leq j \sup_{s\in [0,1]}\int_{A}\frac{d\nu(t)}{|t-s|^{H (d+2\gamma)}},
\end{align*}
Hence \eqref{estim integral} follows by iterating the above estimation, and by using \eqref{ineq Frostman usual}, if we take $$0<\gamma<\dfrac{1}{2}(\kappa/H-d),$$
the right hand side in \eqref{estim integral} is finite. Putting all the above facts together, we arrive at
	\begin{equation}
	\mathbb{E}(\vartheta(x)-\vartheta(y))^{p}\leq C_{p,\gamma}\, \lVert x-y\rVert^{\gamma p},
	\end{equation}
	where $C_{p,\gamma}$ is a constant depends only on $p$ and $\gamma$. Hence we can apply the Kolmogorov's continuity theorem (Theorem 2.3.1 p.158 in \cite{Kh}) to get a continuous version on $\mathbb{R}^d$ of $\vartheta(.)$. Since $A$ is a compact subset of $\left[0,1\right]$ (closed subset
	of $\left[0,1\right]$) and $\left(B^{H}+f\right)$ is continuous
	then $\left(B^{H}+f\right)(A)$ is a compact set in $\mathbb{R}^{d}$.
	Taking into account that $\vartheta$ has a version which is continuous
	in $x$ then it follows from Pitt \cite[ p. 324]{Pi} or Geman and Horowitz
	\cite[p. 12]{GH} that $\left\{ x:\vartheta(x)>0\right\} $ is open,
	non-empty and contained in $\left(B^{H}+f\right)(A)$ almost surely. This ends the proof.
\end{proof}
\begin{remark}
 It is worth noting that the condition $\dim_{\Psi,H}(Gr_{A}(f))>H\,d$ in Theorem \ref{main th 1} is a weaker than $\dim(A)>H\,d$ in Theorem \ref{main th 2}. Indeed let us take $A$ the classical
middle thirds Cantor set which Hausdorff dimension $\dim(A)=\dfrac{\ln\left(2\right)}{\ln\left(3\right)}$.
For $0<\alpha\,d<\dfrac{\ln\left(2\right)}{\ln\left(3\right)}<H\,d$,
we have from $\eqref{parabolic dimension of FBM}$ that $\dim_{\Psi,H}(Gr_{A}(B^{\alpha}))>Hd$. Then by choosing $f$ as a sample function of
$B^{\alpha}$ we obtain 
\[
\dim(A)<H\,d<\dim_{\Psi,H}(Gr_{A}(f)).
\]
Moreover, as can be seen from the proofs of Theorems $\ref{main th 1}$ and $\ref{main th 2}$, we were able to prove the existence of an occupation density through
the conditions $\dim_{\Psi,H}(Gr_{A}(f))>H\,d$ and $\dim(A)>H\,d$ using respectively two different measures namely $\nu$ and $\tau$.  It should be noted that these two measures are constructed by Frostman's lemma. But it was simpler and more practical to work with $\tau$ thanks to the advantageous properties that presents, which are consequences of the condition $\dim(A)>H\,d$, and who have played a key role in the fact that $(B^{H}+f)(A)$  has a non-empty interior. We are therefore naturally led to ask the following question that we have not solved: can we establish the result of Theorem \ref{main th 2} by assuming only that $\dim_{\Psi,H}(Gr_{A}(f))>H\,d$?

\end{remark}


\section{Hölder drifts for which $(B^{H}+f)(A)$ has a non-empty interior a.s.}
The aim of this section is to consider the case $\dim(A)\leq Hd$. Precisely, we seek to construct Hölder functions for which the range $(B^{H}+f)(A)$  has non-empty
interior. Let us make this precise in the following statement.

\begin{theorem}\label{th non-void interior}
Let $\{B^{H}(t): t\in [0,1]\}$ be a $d$-dimensional fractional Brownian motion of Hurst index $H\in (0,1)$. Let $A\subset [0,1]$ a closed set such that $0<\dim(A)\leq Hd$. Then for all $\alpha \in \left( 0,\dim(A)/d\right) $ there exists a $\alpha$-Hölder continuous function $f: [0,1]\rightarrow \mathbb{R}^d$ for which the range $(B^{H}+f)(A)$ has a non-empty interior a.s.
\end{theorem}

The remainder of this section is devoted to the proof of this theorem. To that end, we will make some preparations. Let   $(\Omega^{\prime},\mathcal{F}^{\prime},\mathbb{P}^{\prime})$ be another probability space and $\left\lbrace B^{\alpha^{\prime}}(t):t\in[0,1]\right\rbrace $ be a fractional Brownian
motion with Hurst parameter $\alpha^{\prime}\in \left( 0,\dim(A)/d\right)$ defined on it.  Let us consider the
$d$-dimensional centred Gaussian process $Z$ defined on $(\Omega\times\Omega^{\prime},\mathcal{F}\times\mathcal{F}^{\prime},\mathbb{P}\otimes\mathbb{P}^{\prime})$ by 
\begin{align}
Z(t,(\omega,\omega^{\prime}))=B^{H}(t,\omega)+B^{\alpha^{\prime}}(t,\omega^{\prime}) \text{ for all  } t\in [0,1] \text{ and } (\omega,\omega^{\prime})\in \Omega\times\Omega^{\prime}.
\end{align}
It is easy to see that $Z=(Z_1,...,Z_d)$ where $Z_i$ are independent copies of the real-valued Gaussian process $Z_{0}=B^{H}_{0}+B^{\alpha^{\prime}}_{0}$ with the covariance function given by $$\widetilde{\mathbb{E}}(Z_0(s)Z_0(t))=\mathbb{E}(B^{H}_0(s)B^{H}_0(t))+\mathbb{E}^{\prime}(B^{\alpha^{\prime}}_{0}(s)B^{\alpha^{\prime}}_{0}(t))=\dfrac{1}{2}(s^{2H}+t^{2H}-|t-s|^{2H})+\dfrac{1}{2}(s^{2\alpha^{\prime}}+t^{2\alpha^{\prime}}-|t-s|^{2\alpha^{\prime}}),$$
where $\widetilde{\mathbb{E}}$ and $\mathbb{E}^{\prime}$ denote the expectation under the probability $\widetilde{\mathbb{P}}=\mathbb{P}\otimes \mathbb{P}^{\prime}$ and $\mathbb{P}^{\prime}$ respectively.  

The following proposition is about the local nondeterminism property of the process $Z$. Let $I\subset (0,1]$ be a closed interval, then we have 
\begin{proposition}\label{LND for mixed FBM}
The real-valued process $\{Z_{0}; t\geq 0\}$ satisfy the following
 
 1. For any $s,t \in I$, we have
 \begin{equation}
 \vert s-t\vert ^{2\alpha^{\prime}}\leq \widetilde{\mathbb{E}}\left(Z_0(t)-Z_0(s)\right)^{2}\leq 2\, \vert s-t\vert ^{2\alpha^{\prime}}.
 \end{equation}
 
 2. There exists a positive constant $C$ depending on $\alpha^{\prime}$, $H$ and $I$ only, such that 
 \begin{equation}\label{Slnd 1}
 Var\left(Z_{0}(u)\vert Z_{0}(t_{1}),...,Z_{0}(t_{n})\right)\geq C\left[\min_{0\leq k\leq n}\vert u-t_{k}\vert^{2\alpha^{\prime}}+\min_{0\leq k\leq n}\vert u-t_{k}\vert^{2H}\right],
\end{equation}  
for all integers $n\geq 1$, all $u$, $t_1,...,t_n \in I$ and $t_0=0$.

3. There exists a positive constant $C$ depending on $\alpha^{\prime}$, $H$ and $I$ only, such that for any  $t\in I$ and any $0<r\leq t$,
\begin{equation}\label{slnd 2}
Var\left(Z_0(t)\vert Z_0(s), \vert s-t\vert \geq r\right)\geq C\, r^{2\alpha^{\prime}}.
\end{equation}

\end{proposition}
\begin{proof}
1. For any $s,t \in I$, we have
\[
\begin{array}{lll}
\widetilde{\mathbb{E}}\left(Z_{0}(t)-Z_{0}(s)\right)^{2} & = & \mathbb{E}\left(B^{H}_{0}(t)-B^{H}_{0}(s)\right)^{2}+\mathbb{E}^{\prime}\left(B^{\alpha^{\prime}}_{0}(t)-B^{\alpha^{\prime}}_{0}(s)\right)^{2}\\
\\
& = & \vert t-s\vert^{2H}+\vert t-s\vert^{2\alpha^{\prime}}.
\end{array}
\]
Since $\alpha \leq H$ it follows that 
\[
\vert t-s\vert^{2\alpha^{\prime}}\leq\begin{array}{lll}
\widetilde{\mathbb{E}}\left(Z_{0}(t)-Z_{0}(s)\right)^{2} & \leq & \,2\,\vert t-s\vert^{2\alpha^{\prime}}.\end{array}
\]
2. By definition we can write
\[
\begin{array}{lll}
Var\left(Z_{0}(u)\vert Z_{0}(t_{1}),...,Z_{0}(t_{n})\right) & = & \underset{a_{j}\in\mathbb{R},1\leq j\leq n}{\inf}\widetilde{\mathbb{E}}\left(Z_{0}(u)-\displaystyle\overset{p}{\underset{j=1}{\sum}}a_{j}\,Z_{0}(t_{j})\right)^{2}\\
\\
& = & \underset{a_{j}\in\mathbb{R},1\leq j\leq n}{\inf}\left[\mathbb{E}\left(B^{H}_{0}(u)-\displaystyle\overset{p}{\underset{j=1}{\sum}}a_{j}\,B^{H}_{0}(t_{j})\right)^{2}+\mathbb{E}^{\prime}\left(B^{\alpha^{\prime}}_{0}(u)-\displaystyle\overset{p}{\underset{j=1}{\sum}}a_{j}\,B^{\alpha^{\prime}}_{0}(t_{j})\right)^{2}\right]\\
\\
& \geq &\underset{a_{j}\in\mathbb{R},1\leq j\leq n}{\inf}\mathbb{E}\left(B^{H}_{0}(u)-\displaystyle\overset{p}{\underset{j=1}{\sum}}a_{j}\,B^{H}_{0}(t_{j})\right)^{2}\\
\\
&  & +\underset{b_{j}\in\mathbb{R},1\leq j\leq n}{\inf}\mathbb{E}^{\prime}\left(B^{\alpha^{\prime}}_{0}(u)-\displaystyle\overset{p}{\underset{j=1}{\sum}}b_{j}\,B^{\alpha^{\prime}}_{0}(t_{j})\right)^{2}\\
\\
& = & Var\left(B^{H}_{0}(u)\vert B^{H}_{0}(t_{1}),...,B^{H}_{0}(t_{n})\right)+Var\left(B^{\alpha^{\prime}}_{0}(u)\vert B^{\alpha^{\prime}}_{0}(t_{1}),...,B^{\alpha^{\prime}}_{0}(t_{n})\right).
\end{array}
\]

By the property of strong local nondeterminism of fractional Brownian motion there exists a positive constant $C$ depending on $\alpha$, $H$ and $I$ only, such that for all integers $n\geq 1$, all $u$, $t_1,...,t_n \in I$,
\[
Var\left(Z_{0}(u)\vert Z_{0}(t_{1}),...,Z_{0}(t_{n})\right)\geq C\left[\min_{0\leq k\leq n}\vert u-t_{k}\vert^{2\alpha^{\prime}}+\min_{0\leq k\leq n}\vert u-t_{k}\vert^{2H}\right].
\]
3. We note that, in the Hilbert space setting, the conditional variance
in \eqref{slnd 2} is the squared distance of $Z_{0}(t)$ from the linear subspace spanned
by $\left\lbrace Z_0(s), \vert s-t\vert\geq r\right\rbrace $ in $L^{2}(\widetilde{\mathbb{P}})$. Hence it is sufficient to show that there exists a positive constant $C$ depending on $\alpha^{\prime}$ and $H$ only, such that for all integers $n\geq 1$, $a_{j}\in \mathbb{R}$ and  $s_j \in I$ satisfying $\vert s_{j}-t\vert \geq r,\, (j=1,\cdots,n)$,
$$\widetilde{\mathbb{E}}\left(Z_{0}(t)-\displaystyle\overset{p}{\underset{j=1}{\sum}}a_{j}\,Z_{0}(s_{j})\right)^{2}\geq \,C\, r^{2\alpha^{\prime}}.$$
What follows from \eqref{Slnd 1}. Thus the proof is ended.
\end{proof}

The following is a direct consequence of Corollary 3.3 in \cite{SX}.
 \begin{proposition}\label{non-void interior of mixed FBM}
 $Z(A)$ has $\mathbb{P}\otimes\mathbb{P}^{\prime}$-$a.s.$ non-empty interior.
 \end{proposition}
\begin{proof}[Proof of Theorem \ref{th non-void interior}] Let $\alpha^{\prime}\in  \left( \alpha,\dim(A)/d\right)$. It follows from Proposition \ref{non-void interior of mixed FBM} that $Z(A)=(B^{H}+B^{\alpha^{\prime}})(A)$ has $\mathbb{P}\otimes\mathbb{P}^{\prime}$-$a.s.$ a non-void interior.  Therefore, there exists a $\mathbb{P}^{\prime}$-negligible set $N^{\prime}$ such that, for any $\omega^{\prime}\in N^{\prime\,c}$, there exists a $\mathbb{P}$-negligible set $N$ with for any $\omega\in N^{c}$ the set $(B^{H}(\omega)+B^{\alpha^{\prime}}(\omega^{\prime}))(A)$  has a non-void interior. It is a well known fact that almost all trajectories of $B^{\alpha^{\prime}}$ are $\alpha$-Hölder continuous, and hence the desired drift can be chosen as a sample function of $B^{\alpha^{\prime}}$. 
 \end{proof}
 

\end{document}